\documentclass[12pt,a4paper,oneside]{article}
\usepackage{color,lscape,latexsym,amsmath,amsfonts,amssymb,amscd,eurosym,makeidx}
\usepackage[latin1,ansinew]{inputenc}
\usepackage{hyperref}
\usepackage{setspace}
\usepackage{enumerate}
\usepackage{amsthm}
\usepackage{epsfig} % para incluir figuras .eps
\usepackage{color} % As figuras do xfig precisam...
\usepackage{multicol}
\usepackage[all]{xy} % usar polinomios, etc... muito Ãºtil
\usepackage{graphicx} 
\usepackage{bbm}
\usepackage{mathabx}
\usepackage[british,UKenglish,USenglish,american]{babel}%
\usepackage{tikz-cd}
\usepackage{amsmath,amssymb}
\theoremstyle{plain}
\newtheorem{theorem}{Theorem}%[section] % esta definindo um novo comando, chamado teorema, Ã© numerado continuamente.
\newtheorem{fact}{Fact}
\newtheorem*{claim}{Claim}
\newtheorem*{theoremA}{Theorem A}
\newtheorem*{theoremB}{Theorem B}
\newtheorem*{lemmaA}{Algebric Lemma}
\newtheorem{pro}[theorem]{Proposition}
\newtheorem{lem}[theorem]{Lemma}%[section] % idem...
\newtheorem{remark}[theorem]{Remark}%[section]
\newcommand{\R}{\mathbb{R}}%sÃ£o os atalhos para as letras estilizadas, mathbb sÃ£o os simbolos paranaturais

\DeclareMathOperator{\Ker}{ker}
\DeclareMathOperator{\Im1}{Im}
\DeclareMathOperator{\Span}{span}

\newcommand{\bpr}{\begin{proposition}}
\newcommand{\epr}{\end{proposition}}
\newcommand{\bco}{\begin{corollary}}
\newcommand{\eco}{\end{corollary}}
\newcommand{\blm}{\begin{lemma}}
\newcommand{\elm}{\end{lemma}}
\newcommand{\bdf}{\begin{definition}}
\newcommand{\edf}{\end{definition}}
\newcommand{\bpm}{\begin{pmatrix}}
\newcommand{\epm}{\end{pmatrix}}
\newcommand{\beq}{\begin{equation}}
\newcommand{\eeq}{\end{equation}}
\newcommand{\brm}{\begin{remark}}
\newcommand{\erm}{\end{remark}}
\newcommand{\bcj}{\begin{conjecture}}
\newcommand{\ecj}{\end{conjecture}}
\newcommand{\bqu}{\begin{question}}
\newcommand{\equ}{\end{question}}
\newcommand{\bex}{\begin{exercise}}
\newcommand{\eex}{\end{excercise}}
\newcommand{\bit}{\begin{itemize}}
\newcommand{\eit}{\end{itemize}}
\newcommand{\ben}{\begin{enumerate}}
\newcommand{\een}{\end{enumerate}}

\newcommand\Quotient[2]{
        \mathchoice
            {% \displaystyle
                \text{\raise1ex\hbox{\thinspace $#1$}\Big{/} \lower1ex\hbox{$#2$} \thinspace}%
            }
            {% \textstyle
                #1\,/\,#2
            }
            {% \scriptstyle
                #1\,/\,#2
            }
            {% \scriptscriptstyle  
                #1\,/\,#2
            }
    }

\newcommand\GIT[2]{
        \mathchoice
            {% \displaystyle
                \text{\raise1ex\hbox{\thinspace $#1$}\Big{/}\!\!\!\!\Big{/} \lower1ex\hbox{$#2$} \thinspace}%
            }
            {% \textstyle
                #1\,/\,#2
            }
            {% \scriptstyle
                #1\,/\,#2
            }
            {% \scriptscriptstyle  
                #1\,/\,#2
            }
    }

\begin{document}

\thispagestyle{empty}

%\newpage

%\rule{\textwidth}{1ex} \\ [1ex]

%\begin{minipage}{9,5cm}
%{\Huge Problems section}\\Problems and Solutions
%\end{minipage}

%\vspace{1cm}
%\begin{minipage}{8cm}
%Solutions of Problems\\ List of Problems,\\ Symplectic Geometry
%\end{minipage}

%\vspace{2cm}
%\begin{minipage}{8cm}
%{\bf Primeira EdiÃ§Ã£o}
%\end{minipage}

%\vspace{2cm}%
%\begin{minipage}{8cm}
%{\bf Alcides de Carvalho JÃºnior\\}
%\end{minipage}
%\newpage

\title{Real K\"ahler Submanifolds in Codimension $6$}

\author{Alcides de Carvalho \& Felippe Guimar\~aes \thanks{The first author was supported by CNPq-Brazil,
and the second author by a grant of the CAPES.}}

%\author{Alcides de Carvalho and Felippe Guimar\~aes}
%\author{Felippe Guimar\~aes}
%\tableofcontents

% Revisar preliminares
% Revisar novo remark

\thispagestyle{empty}

\maketitle

\begin{abstract}
\bigskip
We show that a real K\"ahler submanifold in codimension $6$ is essentially a holomorphic submanifold of another real K\"ahler submanifold in lower codimension if the second fundamental form is not sufficiently degenerated. We also give a shorter proof of this result when the real K\"ahler submanifold is minimal, using recent results about isometric rigidity.
\end{abstract}
\section{Introduction}
\bigskip

Let $M^{2n}$ be a K\"ahler manifold of real dimension $2n$. A {\it real K\"ahler submanifold} of codimension $p$ is an isometric immersion $f : M^{2n}\to \R^{2n+p}$. The structure of this submanifolds are sensitive to its codimension. For instance, by \cite{FZKaehlerCylinder}, when the codimension is one and $M^{2n}$ is complete then $f$ must be the product of an isometric embedding of a complete surface $g : \Sigma^2 \rightarrow \R^3$ with the identity map of $\mathbb{C}^{n-1}$. In other words, surfaces in $\R^3$ are essentially the only real K\"ahler submanifolds of codimension one. In codimension two, the situation is also well-understood. Namely, the minimal case was analyzed in details in \cite{DRMinimalKahler}, and the non-minimal case was classified in \cite{FZCod2}. In codimension three, it was proved in \cite{DGKahlerCod3} that unless the submanifold $M^{2n}$ is a holomorphic hypersurface of a real K\"ahler submanifold of codimension one, the complex codimension of the invariant part of the relative nullity, under the almost complex structure of the K\"ahler manifold, has to be less than or equal to $3$. Recently, this result was extended in \cite{YZKahlerCod4} for codimension $4$. They also conjectured that the result holds for higher codimensions. 

Our goal in this paper is twofold. First we give a simple proof of this conjecture when the immersion is minimal and $p \leq 6$ using the modern tools developed in \cite{DFGenuines}. More precisely, we are going to prove the following result.

\begin{theoremA}\label{mainA}
Let $f : M^{2n}\to \R^{2n+p}$ be a real K\"ahler minimal submanifold with $p\leq 6$ and $\nu_f$ be the index of relative nullity of $f$. Then, along each connected component $U$ of an open dense subset of $M^{2n}$, one of the following holds:
\begin{enumerate}[(i)]
    \item $v_f(x)\geq 2n-2p$ for all $x \in U$;
    \item $f$ extends to a real K\"ahler submanifold, that is, there exist a real K\"ahler submanifold $g : N^{2n+2s}\to \R^{2n+p}$ and a holomorphic embedding $h : U \to N^{n+s}$ such that $f|_U = g \circ h$. Moreover, the extension $g$ is also minimal.
\end{enumerate} 
\end{theoremA}

\begin{remark}
A real K\"ahler submanifold is holomorphic if and only if admits a real K\"ahler extension of codimension zero. 
\end{remark}
%Falar que a inclusao da variedade inicial eh holomorfa, ou falar que todas nossas extensoes sao holomorfas

The second main goal is to remove the minimal hypothesis of Theorem A. For this, we use similar arguments from the ones used in \cite{YZKahlerCod4}, that is, we use the complexified second fundamental of $f$ in order to avoid the necessity of the existence of a conjugate immersion of the minimal real K\"ahler submanifold. The goal is to prove the following theorem.

\begin{theoremB}
Let $f : M^{2n}\to \R^{2n+p}$ be a real K\"ahler submanifold with $p\leq 6$ and $\nu_f$ be the index of relative nullity of $f$. Then, along each connected component $U$ of an open dense subset of $M^{2n}$, one of the following holds:
\begin{enumerate}[(i)]
    \item $\dim \Delta \cap J\Delta \geq 2n - 2p$ for all $x \in U$, where $\Delta$ is the relative nullity and $J$ is the almost complex structure of $f$. In particular, $v_f(x)\geq 2n-2p$ for all $x \in U$;
    \item $f$ extends to a real K\"ahler submanifold, that is, there exist a real K\"ahler submanifold $g : N^{2n+2s}\to \R^{2n+p}$ and a holomorphic embedding $h : U \to N^{n+s}$ such that $f|_U = g \circ h$.
\end{enumerate}  
\end{theoremB}
Even though Theorem B implies Theorem A, we have that the proof of Theorem A is quite simpler and is almost a direct consequence of Theorem 14 in \cite{DFGenuines}. On the other hand, the proof of Theorem B will require a structure that is similar to the one constructed in \cite{DFGenuines}.

This paper is organized as follows. In Section \ref{SecPreliminaries}, we establish the
fundamental concepts and notations. In Section \ref{SecMinimal}, we prove Theorem A and, in Section \ref{SecNonMinimal}, we construct specific subbundles in order to obtain Theorem B.

\section{Preliminaries}\label{SecPreliminaries}
In this section we fix the notation and discuss the main tools that will be used in the work. We divide this section in three parts. First, we define the flat bilinear forms that will be used to identify the subbundles necessary to construct the extensions in Theorem A and Theorem B. Second, we study the structure of the tangent and normal bundles of a pair of isometric immersions and their extensions, both given by \cite{DFGenuines}. Finally,  we review some properties about the complexification of the second fundamental form that will be used to obtain the main algebraic lemma required for Theorem B. 

Throughout this section, we assume that $f: M^{2n} \to \R^{2n+p}$ is a real K\"ahler submanifold with almost complex structure $J$.

\subsection{Flat bilinear forms}
Let $W^{p,q}$ be a $(p+q)$-dimensional vector space endowed with a possibly indefinite inner product of signature $(p,q)$, where $q \geq 0$ is the maximal dimension of the subspaces in which the inner product is negative definite. Let $U$ and $V$ be finite dimensional vector spaces. A bilinear form $\gamma\colon U \times V \rightarrow W^{p,q}$ is said to be \textit{flat} if  
\begin{equation}\label{defFlat}
	\left\langle\gamma(X,Z),\gamma(Y,W) \right\rangle - \left\langle
	\gamma(Y,Z),\gamma(X,W)\right\rangle = 0,
\end{equation}
for all $X,Y \in U$ and $Z,W \in V$. Denote the (left) \textit{nullity} of $\gamma$ by $$\mathcal{N}(\gamma) = \{X \in U:\gamma(X,Z)=0 \ \text{for all} \ Z \in V\}$$ and the span of $\gamma$ by $$\mathcal{S}(\gamma) = \text{span}\{\gamma(X,Z):X\in U,Z \in V\}.$$ A subset $S \subset W^{p,q}$ is called \textit{null} if $\left\langle\xi,\eta\right\rangle=0$ for every $\xi,\eta \in S$. We call $X \in U$ a (left) \textit{regular element} of $\gamma$ if $$\dim \gamma^X(V) =\max_{Y \in U}\{\dim \gamma^{Y}(V)\},$$ where $\gamma^Y(Z)= \gamma(Y,Z)$. We will use such elements to obtain the estimates of Theorem B.

Flat bilinear forms were introduced by Moore in \cite{Moore} to study isometric immersions of the round sphere in Euclidean space in low codimension, and can be used to obtain several results about isometric rigidity (see \cite{Base} for more details).

\subsection{The genuine structure}
The work \cite{DFGenuines} provides, in addition to the main results on genuine rigidity,  a strong local geometric structure of any pair of isometric immersions of a given Riemannian manifold. We will make use of such structure in the minimal case.
Assume that $f$ is a minimal isometric immersion and define \begin{equation}\label{oneparemeter}
f_{\theta}(x) = \int_{x_0}^{x} f_*\circ J_{\theta}
\end{equation}
\noindent to be the well-known one-parameter associated family to $f$, where $x_0$ is a fixed point in $M^{2n}$ and $J_\theta = \cos{\theta} I + \sin{\theta}J$ (see \cite{DGGaussMap} or \cite{Base} for more details). We represent by $\widehat{f} = f_{\frac{\pi}{2}}$ the conjugated immersion of $f$ and, by (\ref{oneparemeter}), we can identify $T^{\perp}_fM$ with $T^{\perp}_{\widehat{f}}M$. 

Applying Theorem 11 from \cite{DFGenuines} for the pair $\{f, \widehat{f}\}$ we have a special triple $(\tau, L^l, D^d)$, where $\tau : L \subset T^{\perp}_f M \to \widehat{L} \subset T^{\perp}_{f}M$ is a vector bundle isometry and $D^d = \mathcal{N} (\alpha_{L^{\perp}})\cap \mathcal{N} (\widehat{\alpha}_{\widehat{L}^{\perp}}) \subset TM$ satisfies the following conditions.

\begin{equation}\label{ces}
\left\{
\begin{array}{l}\!(C_1)\;  \mbox{The isometry}
\;\tau\; \mbox{is parallel and preserves second fundamental forms}\\ \mbox{, that is, $\left(\nabla^{\perp}_X \tau \xi\right)_{\widehat{L}} = \tau\left(\nabla^{\perp}_X \xi\right)_{L}$ for all $\xi \in L$, $X \in TM$, and}\\ \mbox{$\tau\circ\alpha_f = \alpha_{\widehat{f}}$;}
\vspace*{1ex}\\
\!(C_2)\;  \mbox{The subbundles}\; L \;\mbox{and} \;\hat L\;
\mbox{are parallel along}\; D^d \;\mbox{in the  normal}\\ \mbox{connections.}
\end{array} \right.
\end{equation}

As in \cite{DFGenuines}, the bilinear form $\phi_{\tau}: (TM \oplus L) \times TM \to L^{\perp}\times\widehat{L}^{\perp}$ given by
\begin{equation}\label{eq1}
\phi_{\tau}(Y +\xi, X) =\left(\left(\tilde{\nabla}_X(Y + \xi)\right)_{L^{\perp}},\left(\tilde{\nabla}_X(JY + \tau\xi)\right)_{\widehat{L}^{\perp}} \right), 
\end{equation}
\noindent plays a key role in the construction of the isometric extensions of $f$ and $\widehat{f}$. Observe that condition $(C_1)$ is equivalent to the flatness of $\phi_{\tau}$.

We give an idea on how the extensions of $f$ and $\widehat{f}$ are constructed. Choose any smooth rank $s$ subbundle $\Lambda \subset TM \oplus L$, and define the maps $F_{\Lambda,f}: \Lambda \rightarrow \R^{2n+p}$ and $\widehat{F}_{\Lambda,\widehat{f}}:\Lambda \rightarrow \R^{2n+p}$ by \begin{equation}\label{extension}
    F(v) = f(x) + v,\quad \widehat{F}(v) = \widehat{f}(x) + (J \oplus \tau)v,\quad v \in \Lambda_x,\, x \in M^{2n}.
\end{equation} 
\noindent We say that $F_{\Lambda,f}$ is a \textit{extension} of $f$ if it is an immersion in some open neighborhood of the $0$-section of $\Lambda$. In particular, if $\Lambda$ is transversal to $TM$ then $F_{\Lambda, f}$ is an extension. By Proposition 12 in \cite{Belezinha}, we have that the two maps $F_{\Lambda,f}$ and $\widehat{F}_{\Lambda,\widehat{f}}$ are isometric if and only if $\phi_{\tau}(\Lambda,TM)$ is a null subspace. By Proposition 9 in \cite{DFGenuines}, we have that $f$ and $\widehat{f}$ are mutually $D^d$-ruled when there is no such extensions. By $D^d$-ruled (or just $d$-ruled) we mean that the leaves of the distribution $D^d$ (a $d$-dimensional distribution) are mapped diffeomorphically by $f$ to (open subsets of) affine subspaces of $\R^{n+p}$. 

Although we can extend the pair $\{f, \widehat{f}\}$, we are just interested in the extensions of $f$ and its almost complex structure to obtain K\"ahler extensions. For the sake of clarity, we will enunciate a version of Proposition $9$ in \cite{DFGenuines} for a real Kahler submanifold. Let $\mathcal{T}: L^l \rightarrow L^l$ be a vector bundle isometry that satisfies conditions $(C_1)$ and $(C_2)$ for $D = \mathcal{N}(\alpha_{L^{\perp}})$. By Lemma 8 in \cite{DFGenuines}, the distribution $D \subset \mathcal{N}(\phi_{\mathcal{T}})$ is integrable and $\mathcal{N}(\phi_{\mathcal{T}}) \cap TM = D$ holds. Define $F = F_{\Lambda,f}$, where $\Lambda$ is given by $\mathcal{N}(\phi_{\mathcal{T}}) = D \oplus \Lambda$. Then, we have the following proposition.

\begin{pro}\label{newProp9}
The isometric immersion $F$ is a $\mathcal{N}(\phi_{\mathcal{T}})$-ruled real K\"ahler extension of $f$. Moreover, there is a smooth orthogonal splittings $T^{\perp}_FN = \mathcal{L}\oplus L^{\perp}$ and a vector bundle isometry $T: \mathcal{L} \rightarrow \mathcal{L}$ such that $\mathcal{N}(\phi_{\mathcal{T}}) = \mathcal{N(\alpha^{F}_{\mathcal{L}^{\perp}})}$, and the triple $(T, \mathcal{L}, \mathcal{N}(\phi_{\mathcal{T}}))$ satisfies conditions $(C_1)$ and $(C_2)$ in (\ref{ces}).
\end{pro}
\begin{proof}
It is a corollary from Proposition 9 in \cite{DFGenuines} once observed that $\Lambda$ is $\mathcal{T}$-invariant and the almost complex structure of $\Lambda$ is the parallel isometry $(J\oplus\mathcal{T})|_{\Lambda}: \Lambda \rightarrow \Lambda$ (in the induced connection).
\end{proof}
%When $\Lambda$ is transversal to TM (remark ?) we can apply Proposition 9 in [2] to
\subsection{The structure of the complexified second fundamental form}
Our objective is to show that, at a generic point $x \in M^{2n}$, the second fundamental form takes a rather special form. We shall begin with the following definitions.

We denote by $\Delta(x)$ the {\it relative nullity} of $f$ at $x,$ that is, $$\Delta(x)  = \{Z \in T_xM: \alpha(Z,Y) = 0 \,\, \text{for all } Y \in TM\},$$ and by $\nu_f(x)$ the index of the relative nullity of $f$ at $x,$ i.e., $\nu_f(x) = \dim_{\R} \Delta(x).$ Let us define the complex subspace $\Delta_0\subset T_xM$ by $\Delta_0(x) = \Delta(x)\cap J\Delta(x),$ and denote its complex dimension by $\nu'(x) = \dim_{\mathbb{C}} \Delta_0(x)$. Define the {\it pluriharmonic nullity} $\Delta_J$ of $f$ by $$\Delta_J := \{Z \in TM: \alpha(Z,JY) = \alpha(JZ,Y), \quad\mbox{for all}\quad Y \in TM\},$$ and observe that $\Delta_0 = \Delta\cap \Delta_J$.

We can naturally extend $J$ to a linear operator on the complexified tangent space $T_xM \otimes \mathbb{C}$ then let $V$ be the eigenspaces associated to the eigenvalue $i$ of $J,$ thus $V$ is the complex subspace of $T_xM \otimes \mathbb{C}$ defined as $V = \{X-iJX:X \in T_xM\}.$ Denote  by $W$ the complex linear subspace of $V$ perpendicular to $\Delta_0,$ that is, $W\oplus\overline{W} =  \Delta_0 ^{\perp}\otimes \mathbb{C}.$

Let us now recall the following decomposition of the second fundamental form $\alpha$ of $f$ at $x \in M$. Extend $\alpha$ bilinearly over $\mathbb{C}$ and the inner product in $T^{\perp}M$ bilinearly over $\mathbb{C}$, and still denote it by $\alpha,$ 
$$\alpha: T_xM \otimes \mathbb{C}\times T_xM \otimes \mathbb{C}\to T_x^{\perp}M \otimes \mathbb{C}.$$
Using that $T_xM \otimes \mathbb{C} = V \oplus \overline{V},$ we can write $$H = \alpha |_{V\times \overline{V}} \quad \mbox{and}\quad S = \alpha|_{V\times V}$$ for the $(1,1)$ and $(2,0)$ parts of $\alpha$, respectively. Let $W' \subset V$ be the complex linear subspace given $W' = \ker H \cap \ker S.$ Hence, $\Delta_0\otimes \mathbb{C} =  W'\oplus \overline{W'}$ and $V = W \oplus W'.$ Notice that $\ker H =  \Delta_J$ and, by Lemma 7 in \cite{FHZKahler}, we have $$\dim \Delta_J \geq 2n - 2p.$$  
%\mathcal{N}(H) =
As observed in \cite{FHZKahler}, the K\"ahlerness of $M$ implies that the Hermitian bilinear form $H$ and the symmetric bilinear form $S$ satisfy the following symmetry conditions

\begin{align}\numberwithin{equation}{section}
\left\langle H(X,\overline{Y}),H(Z,\overline{W}) \right\rangle & = \left\langle H(Z,\overline{Y}),H(X,\overline{W}) \right\rangle \label{eq2.1}\\
\left\langle H(X,\overline{Y}),S(Z,W) \right\rangle & = \left\langle H(Z,\overline{Y}),S(X,W) \right\rangle \label{eq2.2}\\
\left\langle S(X,Y),S(Z,W) \right\rangle & =  \left\langle S(Z,Y), S(X,W) \right\rangle\label{eq2.3}
\end{align}
 
for any $X,Y, Z, W \in V$.

% VERIFICAR EM GENUINAS SE A ESTRUTURA NAO IMPLICA DIRETAMENTE L=L CHAPEU
% VER SE NO SINGULARES L_D JA NAO RESOLVE
 
\section{Minimal Case}\label{SecMinimal}
The aim of this section is to prove Theorem A. Since $f$ is minimal, as discussed in Section \ref{SecPreliminaries}, there is a pair of isometric immersions $\{f,\widehat{f}\}$ and a triple $(\tau, L^l, D^d)$ that satisfies the conditions $(C_1)$ and $(C_2)$. We have the following result.

\begin{lem}\label{lemma1} Let $E \subset L$ be the maximal invariant subspace by $\tau$. Then, the dimension of $E$ is even and $D=\mathcal{N} (\alpha_{L^{\perp}})\cap \mathcal{N} (\widehat{\alpha}_{\widehat{L}^{\perp}}) = \mathcal{N} (\alpha_{E^{\perp}}).$
\end{lem}

\begin{proof}
It follows directly from the involution propriety of $\tau$ that the dimension of $E$ is even. Since $E \subset L \cap \widehat{L},$ we have that $(L\cap \widehat{L})^{\perp} \subset E^{\perp}.$ So, it is clear that,  $\mathcal{N} (\alpha_{E^{\perp}}) \subset \mathcal{N} (\alpha_{L^{\perp}})\cap \mathcal{N} (\widehat{\alpha}_{\widehat{L}^{\perp}}).$ 
Consider the orthogonal decomposition $E^{\perp} =L^{\perp}\oplus V$, where $V$ is the orthogonal complement of $E$ in $L$. Take an unit vector $\xi \in E^{\perp}$ and define the following sequence $\xi \;=\; \xi_0^{V} +\xi_0^{L^{\perp}}$ and $\tau(\xi_i^{V})\;=\; \xi_{i+1}^{V} + \xi_{i+1}^{L^{\perp}}$, where  $\xi_{i+1}^{V}$ is the $V$-component and $ \xi_{i+1}^{L^{\perp}}$ is the $L^\perp$-component of $\tau(\xi_i^{V})$. 

Note that the norm of the sequence $\xi_{i}^{V}$ satisfies $$0\leq \ldots \leq  \|\xi_{i+1}^{V}\|\leq \|\xi_{i}^{V}\|\leq\ldots\leq\|\xi_{0}^{V}\|\leq 1.$$ Therefore, any convergent subsequence of $\xi^{V}_i$ converges to a vector with norm equal to $\lim_{i\to\infty}\|\xi_{i}^{V}\|$. Then, up to a subsequence, $\xi^{V}_{i}$ converges to $w\in V$ and so $\tau (\xi^{V}_{i}) = \xi ^{V}_{i+1}+\xi ^{L^{\perp}}_{i+1}$ converges to $w' \in V$. In fact, since $\tau$ is a isometry and $\xi^{V}_{i}, \xi ^{V}_{i+1}$ converges to vectors of the same norm then the $L^{\perp}$ component of $\tau (\xi^{V}_{i})$ converges to $0$. Then it is easy to see that $w=w'=0$, otherwise $E\oplus \Span\{w,w'\}$ is invariant by $\tau$ which contradicts the maximality of $E$.

Let $X \in \mathcal{N} (\alpha_{L^{\perp}})\cap \mathcal{N} (\widehat{\alpha}_{\widehat{L}^{\perp}})$ be an unit vector. By definition of the sequence we have that $$A_{\xi}X =A_{\xi_0}X =-JA_{\tau\xi_0^{V}}X =  -JA_{\xi_1^{V}}X= A_{\tau\xi_1^{V}}X = \ldots
= \pm -JA_{\xi_i^{V}}X,$$ for all $i \in \mathbb{N}$. By taking the limit it follows that $A_{\xi}X = 0$ for all $\xi \in E^{\perp}$ and $X \in \mathcal{N} (\alpha_{L^{\perp}})\cap \mathcal{N} (\widehat{\alpha}_{\widehat{L}^{\perp}})$ and therefore $\mathcal{N} (\alpha_{L^{\perp}})\cap \mathcal{N} (\widehat{\alpha}_{\widehat{L}^{\perp}})\subset\mathcal{N} (\alpha_{E^{\perp}})$ which finishes the proof of the lemma.

\end{proof}

\begin{remark}\label{remark1}
Note that the maximal invariant $E$ is given by $L \cap \tau(L)$ and, in particular, it has dimension at least $2l-p.$
\end{remark}

Lemma \ref{lemma1} implies that $(\overline{\tau} = \tau|_E ,E^{2k}, D^d)$ also satisfies the conditions $(C1)$ and $(C2)$ in (\ref{ces}). Then by Proposition 9 in \cite{DFGenuines} we have an isometric extension $F_{\Lambda,f}$ (possibly trivial), where $\Lambda$ is given by the orthogonal splitting $\mathcal{N}(\phi_{\overline{\tau}}) = D \oplus \Lambda$ and is trasversal to $TM$. Note that $J\oplus\overline{\tau}  :  TM\oplus E\to TM \oplus E$ is an isometry and parallel in the induced connection. Thus, if $f$ extend isometrically then the extension is K\"ahler. Otherwise, $f$ is $D^d$-ruled and from remark $20$ in \cite{DFGenuines} we have $d \geq 2n -2p +2l$. Moreover, if $E=L$ and $p \leq 6$ (See Remark \ref{remark2}) it follows from Theorem $14$ in \cite{DFGenuines} a better estimative, namely, $d \geq 2n -2p +3l$.

\begin{remark}\label{remark2}
Theorem $14$ in \cite{DFGenuines} and Lemma \ref{lemma1} implies that if $p=6$ and $l = 0$ then $D^d$ is the relative nullity and $d \geq 2n -2p -1$. Furthermore, since $D^d$ is $J$-invariant then $d$ is even and so, in fact, $d \geq 2n -2p$.
\end{remark}

The discussion above yields the following proposition.

\begin{pro}\label{pps1}
Let $f: M^{2n} \to \R^{2n+p}$ be a real K\"ahler minimal submanifold with $p \leq 6$ and $\widehat{f} = f_{\pi/2}: M^{2n}\to \R^{2n+p}$ its conjugate. Then, the pair $\{f, \widehat{f}\}$ are either mutually $D^{d}$-ruled with $d\geq 2n-2p+2l$ or extends to a real K\"ahler submanifold (locally). Moreover, if $E=L$ and the pair $\{f, \widehat{f}\}$ does not extend, then $d \geq 2n-2p+3l.$
\end{pro}

\noindent We now give a complete proof of Theorem A using the triple $(\tau, E^{2k},D^d)$ above.

\begin{proof}[Proof of theorem A]
It follows from Proposition \ref{pps1} that, along each connected component of an open dense subset of $M^{2n}$, either $f$ and $\widehat{f}$ extend isometrically to a real K\"ahler submanifold, or $\{f,\widehat{f}\}$ are mutually $d$-ruled, with $d \geq 2n - 2p + 2l$. If the first case occurs the conclusion of the theorem is immediate. So we assume that $\{f,\widehat{f}\}$ are mutually $D^{d}$-ruled.

Since $D^d$ is an asymptotic space we have that $A_\xi|_D (D) \subset  D^{\perp}$ for any $\xi \in T_f^{\perp}M$ and therefore \begin{equation}\label{eq:assymptoticShape}\dim (\ker A_\xi \cap D)\geq 2n+4l-4p,\end{equation} for all $\xi \in T_f^{\perp}M$. Also, from condition $(C_1)$ we have that $\Ker A_{\xi} = \Ker A_{\tau\xi}$ for all $\xi \in E^{2k}$, where $E^{2k}$ is the maximal invariant space of $\tau$ as in Lemma \ref{lemma1}. Thus, the rellative nullity $\Delta$ of $f$ is the intersection of $D^d$ and $\bigcap_{i=1}^{k}\Ker A_{\xi_i}$ for some basis $\{\xi_1,\cdots, \xi_k,$ $\tau\xi_1, \cdots \tau\xi_k\}$ of $E^{2k}$. Combining the inequalities \eqref{eq:assymptoticShape} for $\xi_1,\cdots,\xi_k$, we have that $$\nu_f \geq d - k(2p-2l) \geq (2n-2p)+2l(k+1)-2pk.$$ Thus, the theorem follows if we prove that $$l(k+1)-pk \geq 0.$$ This easily holds for $p \leq 4$. For $p \in \{5,6\}$ the inequality fails if $l = 2k = 2$. In this case we can apply the estimate $d \geq 2n -2p +3l$ since $L=E$ (see remark \ref{remark2}), and conclude in an analogous way that $$\nu_f \geq d - k(2p-3l) \geq (2n-2p)+3l(k+1)-2pk \geq 2n -2p.$$
\end{proof}

\begin{remark}
If the real K\"ahler submanifold in Theorem A is complete then the rules obtained by Theorem \ref{GenuineKahler} are holomorphic distributions and they are naturally related with the appearance of the complex ruled submanifolds as in \cite{DRMinimalKahler} and \cite{FZKaehlerCylinder} (see \cite{Base} for more details).
\end{remark}

%We will need a definition to state the corollary of Theorem A for complete submanifolds. A real K\"ahler submanifold is said \textit{$s$-complex ruled}, if  admits a continuous foliation of complex dimension $s$ such that any leaf is a holomorphic submanifold of  $M^{2n}$ whose image by f is part of an affine subspace of the ambient space. If, in addition, all leaves are complete Euclidean spaces $\R^{2s}$, then $f$ is said to be completely complex ruled. 

%\begin{cor}\label{corComplete}
%Let $f: M^{2n} \rightarrow \R^{2n+p}$ be a complete minimal real K\"ahler submanifold with $n>p$ and $p \leq 6$. Then either $f$ extends to a real minimal real K\"ahler submanifold, or is a complete complex ruled submanifold.
%\end{cor}
%\begin{proof}
%We may assume that $M^{2n}$ is simply connected because, otherwise, we can argue for its universal cover. By the proof of Theorem A we have that BLA BLA
%\end{proof}

%Fazer REMARK dizedo que Genuinas + Lemma algebrico implica teorema B da mesma forma que implica A
% Verificar main lemma para cod p=6 (testar casos)
% Definir projecos S barra, etc...
% Fazer remark mencionar que a modificacao nao eh tao simples

%Let $U \subset M$ be the open set where $\nu$ attains its minimum value $\nu_0$,
% Talvez remark sobre a forma que extendemos, extensao kahler holomorfa
%$$U = \{x \in M | \,\nu(x) = \nu_0\}.$$ Let $\nu'_0$ be the minimum value of $\nu'(x)$ for all $x \in U,$ and $U_0\subset U$ be the open subset where $\nu' = \nu'_0$.

%Mencionar a extensao do zheng

\section{Non-minimal case}\label{SecNonMinimal}
For the non-minimal case we are not able to apply directly the main result in \cite{DFGenuines}. However we are still able to obtain the whole structure of bundles constructed there.

Let $f:M^{2n} \rightarrow \R^{2n+p}$ be a real K\"ahler submanifold with second fundamental form $\alpha$ and endow the vector bundle $T^{\perp}_fM \oplus T^{\perp}_fM$ with the indefinite metric of type $(p,p)$ given by $$\left< \left<\, ,\, \right> \right>_{T_f^{\perp}M\oplus{T_f^{\perp}M}} = \left<\, ,\,\right>_{\pi_1T_f^{\perp}M} - \left<\,,\,\right>_{\pi_2T_f^{\perp}M},$$ where $\pi_1$ and $\pi_2$ are the natural projections of $T^{\perp}_fM \oplus T^{\perp}_fM$. Set $$\beta=\alpha(\cdot,\cdot)\oplus\alpha(J\cdot,\cdot):TM \times TM \rightarrow \mathcal{S}(\alpha)\oplus \mathcal{S}(\alpha),$$ and let $\Omega \subset \mathcal{S}(\alpha) \oplus \mathcal{S}(\alpha) = \mathcal{S}(\alpha) \oplus \mathcal{S}(\alpha(J\cdot,\cdot))$ be the vector bundle with null fibers $\Omega = \mathcal{S}(\beta) \cap \mathcal{S}(\beta)^{\perp}$. Observe that $\Omega$ is an even dimensional vector bundle by the almost complex structure $J$. Indeed, if $\beta(X,Y) \in \Omega$ and is not zero for some $X,Y\in TM$ then $\beta(JX,Y) \in \Omega$ and $\beta(X,Y), \beta(JX,Y)$ are linearly independents. Accordingly, there is an orthogonal splitting $$\mathcal{S}(\alpha) = \Gamma \oplus \Gamma^{\perp}$$ where $\Gamma = \mathcal{S}(\alpha) \cap \Omega^{\perp}$, and an isometry $\mathcal{T}: \Gamma^{\perp} \rightarrow \Gamma^{\perp}$ such that $$\Omega = \{(\eta, \mathcal{T}\eta): \eta \in \Gamma^{\perp}\} \subset \Gamma^{\perp}\oplus \Gamma^{\perp}$$ and $\alpha_{\Gamma^{\perp}}(\cdot,\cdot) =\mathcal{T} \circ  \alpha_{\Gamma^{\perp}}(J\cdot,\cdot)$. Furthermore, $$\mathcal{T}^2 \circ \alpha_{\Gamma^{\perp}}(X,Y) = \mathcal{T} \circ \alpha_{\Gamma^{\perp}}(JX,Y) = -\alpha_{\Gamma^{\perp}}(X,Y).$$ In particular, $\mathcal{T}^2 = - I$.

Define $\beta_{\Gamma}: TM \times TM \rightarrow \Gamma \oplus \Gamma$ as the projection of $\beta$ over ${\Gamma}$ and a vector subbundle $\Theta \subset TM$ by $\Theta = \mathcal{N}(\beta_{\Gamma})$. The vector subbundle $S \subset \Gamma^{\perp}$ defined by $$S = \mathcal{S}(\alpha|_{\Theta \times TM})$$ is preserved by $\mathcal{T}$, by the $J$-invariance of $\Theta$, and satisfies $\Theta = \mathcal{N}(\alpha_{S^{\perp}}) \cap \mathcal{N}(\alpha_{S^{\perp}}(J\cdot,\cdot))$. Now define a vector subbundle $S_0 \subset S$ by $$S_0 = \bigcap_{X \in TM} \Ker \mathcal{K}(X),$$ where $\mathcal{K}(X) \in \Lambda^2(S)$, $X \in TM$ denotes the skew-symmetric tensor given by $$\mathcal{K}(X){\eta} = (\nabla^{\perp}_X\eta)_S - (\nabla^{\perp}_X\tau\eta)_S.$$ Then we define vector subbundles $L^l \subset S_0$ and $D^d \subset \Theta$ by $$L^l = \{\delta \in S_0 : \nabla^{\perp}_Y \delta \in S \,\, \text{and } \nabla^{\perp}_Y \mathcal{T}\delta \in S \,\, \text{for all } Y\in \Theta\}$$ and $D^d = \mathcal{N}(\alpha_{L^{\perp}}) \cap \mathcal{N}(\alpha_{L^{\perp}}(J\cdot,\cdot))$, and let $\tau: L^l \rightarrow L^{l}$ be the induced vector bundle isometry given by $ \tau = \mathcal{T}|_{L^l}$. In particular, $\tau^2 = -I$, where $I$ is the identity map and $l = 2k$ for some integer $k\geq0$. We can now state an K\"ahler version of Theorem 11 in \cite{DFGenuines}.
\begin{theorem}
Let $f:M^{2n} \to \R^{2n+p}$ be a real K\"ahler submanifold. Then along each connected component of an open dense subset of $M^{2n}$ the triple $(\tau, L^l, D^d)$, satisfies $(C_1)$ and $(C_2)$ in (\ref{ces}). In particular, $f$ has an isometric K\"ahler ruled extension $F:N^{2n+2k} \rightarrow \R^{2n+p}$ (possibly trivial) satisfying the conclusions of Proposition \ref{newProp9}. Moreover, if the extension is trivial then $f$ is $D^d$-ruled.
\end{theorem}
\begin{proof}
The proof is the same as the proof of Theorem 11 in \cite{DFGenuines} once you notice that there is no need for $\beta$ to be symmetric and Lemma 12 in \cite{DFGenuines} holds if we use the following propriety $R(X,Y) = R(JX,JY)$ for all $X,Y \in TM$.
\end{proof}

Observe that, at first, we do not have any estimate for the dimension of $D^d$ and the theorem above holds for any codimension (the subbundles can be trivial). Even though we do not have the conjugated isometric immersion as in Theorem A, we have that the flat bilinear form $\beta$ is special. The reason is that the only lemma that requires symmetry in \cite{DFGenuines} holds for $\beta$, more precisely, we have the following algebraic lemma. In this lemma we will use the nomenclature of kernel and image ($\ker$ and $\Im1$) instead of nullity and span ($\mathcal{N}$ and $\mathcal{S}$) to follow the same notation of \cite{YZKahlerCod4}.

%Until now, we only used the minimality of $f$ to ensure the existence of the conjugated isometric immersion and to be able to apply the work \cite{DFGenuines} for the pair $\{f,\widehat{f}\}$. Indeed, to obtain the estimate of the dimension of the relative nullity estimate of $f$ we need the main algebraic tool (Theorem 3) of \cite{DFCompositions} for symmetric flat bilinear forms, more precisely, for $\beta = \phi_0$ given by (\ref{eq1}). It is well know that the flat bilinear form $\beta$ is symmetric if and only if $f$ is minimal, in other words, in order to prove Theorem B we need a new algebraic tool for nonsymmetric flat bilinear forms. For this, Let $H$ and $S$ be the parts $(1,1)$ and $(2,0)$ of the complexification of the second fundamental form of $f$ as discussed in Section \ref{SecPreliminaries}. Following the ideas of Zheng-Yan \cite{YZKahlerCod4} we are able to prove their main algebraic lemma for higher codimension. Namely:
% Preliminares sec 2.3
% Lemma e Theo B
% S_E eh a projecao

\begin{lemmaA}

Let $V\simeq \mathbb{C}^{n}$ and $N\simeq \R^p$ be equipped with inner products, and let $H$ and $S$ be (respectively) Hermitian and symmetric bilinear forms from $V$ into $N_{\mathbb{C}} = N \otimes \mathbb{C}$ satisfying symmetry conditions (\ref{eq2.1})-(\ref{eq2.3}). Then, there exist a subspace $E$ of $(\Im1 H)^{\perp}\subset N$ (possibly trivial) with real dimension $2k$, and an isometry $\tau$ of $E$ onto itself such that $\tau^{2} = -I$ and $\langle S(\cdot,\cdot),\tau\eta\rangle =-i \langle S(\cdot,\cdot),\eta\rangle$ and $\dim_{\mathbb{C}} \ker H \cap \ker S_{E^{\perp}}\geq n-p+k.$
\end{lemmaA}

\begin{proof}
When $\Im1 H = \{0\}$, the lemma follows from Theorem 3 in \cite{DFCompositions} since $S$ is a flat symmetric bilinear form. In fact, this theorem say that $\langle S(\cdot,\cdot),\tau\eta\rangle =-i \langle S(\cdot,\cdot),\eta\rangle$ for all $\eta \in E$ and $\dim_{\mathbb{C}} \ker S_{E^{\perp}}\geq n-p+k$ where $E \otimes \mathbb{C}$ is the null part of $S$. Notice that if $k=0$ we actually have the inequality by the fact that $E$ is even dimensional (for $p=6$, as discussed in Theorem A). %$\dim \ker S \geq 2n-2p-1$, but the real dimension of $\ker S$ is even and the inequality  $\dim \ker S \geq 2n-2p$ holds.

Since $H$ is Hermitian, its image space $\Im1 H$  is a tensor product of some real linear subspace of $N$ which we denote also by $\Im1 H$ and $\mathbb{C}$. Let $N_{\mathbb{C}} = (\Im1 H \oplus \Im1 H^{\perp}) \otimes \mathbb{C}$ be the orthogonal (possibly trivial) decomposition and write $S = (S', S'')$ under this decomposition. Suppose that $\dim \Im1 H =1$, in this case it is easy to see that $S'$ is flat by equation (\ref{eq2.2}). Since $S$ is already flat, it follows that $S''$ is also flat. Thus, applying Theorem 3 in \cite{DFCompositions} for $S''$ the lemma follows as above.

% By Lemma 3 in \cite{DGKahlerCod3} (see Remark ?) for flat bilinear forms (non-symmetric) applied to
We can assume now that $6 \geq \dim \Im1 H \geq 2$. By equation (\ref{eq2.2}) we have that $S(\ker H, V)\subset(\Im1 H)^{\perp} \otimes \mathbb{C}$. Define the flat bilinear form $\overline{S}$ by $$\overline{S}=S|_{\ker H \times V}: \ker H \times V \to (\Im1 H)^{\perp}\otimes \mathbb{C}.$$ The lemma will follows from Lemma 7 in \cite{FHZKahler} and the following claim.%We have that
\begin{claim}  $(\Im1 H)^{\perp}$ possesses a direct sum decomposition $(\Im1 H)^{\perp} = E\oplus E^{\perp}$, such that $\overline{S}_{E}$ is null and $\overline{S}_{E^{\perp}}$ is flat and satisfies $\dim \Ker(\overline{S}_{E^{\perp}})\geq \dim \ker H -\dim E^{\perp}$.
\end{claim}
In fact, since $\dim_{\mathbb{C}}\Ker H \geq n-\dim_{\mathbb{C}} \Im1 H$, by Lemma 7 in \cite{FHZKahler}, it follows from Claim that $\dim_{\mathbb{C}} \overline{S}_{E^\perp} \geq \dim_{\mathbb{C}} \ker H -p + (\dim_{\mathbb{C}} \Im1 H + k)\geq n-p+k$ and $E$ is trivial.

Now we will prove the claim. We can assume that $\dim_{\mathbb{C}} \ker{\overline{S}} < n-p$. Otherwise, by Lemma 7 in \cite{FHZKahler}, it follows that $\dim_{\mathbb{C}} \Ker H \cap \Ker S \geq n - p$ and there is nothing to do.
\begin{fact}
If $X \in \ker H$ is a regular element, then the restriction of $\langle ,\rangle$ to $\overline{S}^X(V) = \Im1 \overline{S}(X,\cdot)$ is degenerate.
\end{fact} 
 \noindent Otherwise $\overline{S}^{X}(V) \cap (\overline{S}^{X}(V))^{\perp} = \{0\}$ and it follows that $\ker \overline{S}^{X} \subset \ker \overline{S}$. Since, by definition, $\ker \overline{S} \subset\ker\overline{S}^{X}$, we conclude $\dim_{\mathbb{C}} \ker \overline{S} = \dim_{\mathbb{C}} \ker \overline{S}^{X} \geq n-p+k,$ which is a contradiction.
 \begin{fact}
If $\overline{S}^{X}(V)$ is a null subspace of $(\Im1 H)^{\perp}$ for all regular element $X \in \ker H$, i.e., $\langle , \rangle \equiv 0$ restricted to  $\overline{S}^{X}(V)\times  \overline{S}^{X}(V)$, for all regular element $X \in \ker H$. Then the Claim holds.
\end{fact}
Indeed, $\langle \overline{S}(X , Z) , \overline{S}(X , W)\rangle= 0$, for all $Z , W \in V.$ Since the set of regular elements is dense, we have by continuity that $\langle\overline{S}(X , Z) , \overline{S}(X , W)\rangle= 0$, for all $ Z , W \in V$ and $X \in \ker H$ . By flatness $$0 = \langle\overline{S}(X + Y, Z), \overline{S}(X +Y, W)\rangle = 2(\overline{S}(X, W),\overline{S}(Y, Z)),$$
for all $Z, W \in V$ and $X, Y \in \ker H.$ Setting $E = (\Im1 H)^{\perp}$, we obtain the conclusions of the affirmation in this case.

By Fact 2, we can assume that there exists a regular element $X \in \ker H$ such that $\overline{S}^{X}(V)$ is not a null subspace. Using Fact 1, we have that the subspace $$U^{X} = \overline{S}^{X}(V) \cap (\overline{S}^{X}(V))^{\perp}$$ satisfies $ 1\leq \dim_{\mathbb{C}} U^X\leq 3$. Using the almost complex structure it is not hard to see that  $U^X$ is a complex vector space and we can write $U^X = E \otimes \mathbb{C}$, where $E$ is an even real dimensional subspace. By this discussion we have that $\dim_{\mathbb{C}} U^X = 2.$ %In fact, $\beta(X,Y_0)=(\xi,\bar{\xi})\in U^{X}$ if and only if $A_\xi X = -JA_{\bar{\xi}}X$, thus is easy to see that $\beta(X,JY_0)=(\bar{\xi},-\xi)\in U^{X}$ and $\beta(X,Y_0)$ $\beta(X,JY_0)$ are linearly independent.

%We have $\mathcal{S}(\beta|_{V\times \ker\beta^X}) \subset U^X$ by lemma (\ref{Lemma1}), so the idea here is that if $U^X$ is low dimensional then there are $L\subset \mathcal{N}(\beta)$ with $\dim L \geq 2n-2p+k.$

Using the flatness of $\overline{S}$ and that $\dim U^X =2$  we can see that $U^{X} $ is orthogonal to $\Im1 \overline{S}$. Thus, $\Im1 \overline{S} \cap \Im1 \overline{S}^{\perp} = U^X $ and the projection of $\overline{S}$ in $E^{\perp}\oplus \mathbb{C}$ is non-degenerate. In particular, $\dim_{\mathbb{C}} \overline{S}_{E^\perp} \geq \dim_{\mathbb{C}} \ker H -p + (\dim_{\mathbb{C}} \Im1 H + 1)$ and the claim holds.

\end{proof}

\begin{remark}\label{rmk:teoCod3}
We could have used Lemma 3 in \cite{DGKahlerCod3} for $\overline{S}$ when $\dim H \geq 3$ to shorten the proof of our Algebraic lemma. The proof of their lemma is not immediate as they mentioned and for this reason we prefer to make a more self-contained version.
\end{remark}

\begin{remark}\label{onlyBeta}
This Algebraic Lemma makes strong use of the complexification of the second fundamental form and has already been proved up to codimension 4 in \cite{YZKahlerCod4}. In general, Theorem 3 in \cite{DFCompositions} does not necessary holds for an arbitrary flat bilinear form.
\end{remark}

Note that the $(0,2)$ and $(1,1)$ parts of the complexified second fundamental form of $f$ contain the same algebric information of $\beta$. In fact, our Algebric Lemma is an alternative version of Theorem 3 in \cite{DFCompositions} for $\beta$ (remark \ref{onlyBeta}) up to codimension $6$.

Now we give a version of Theorem 14 in \cite{DFGenuines} that is a fundamental tool in the proof of Theorem B.

\begin{theorem}\label{GenuineKahler}
Let $f: M^{2n} \rightarrow \R^{2n+p}$ be a real K\"ahler submanifold with $2p<n$ and $p \leq 6$. Then either $f$ is $D^d$-ruled with $d \geq 2n - 2p +3l$ or $f$ extends to a real K\"ahler submanifold along each connected component of an open dense subset of $M^{2n}$.  
\end{theorem}
\begin{proof}
The proof is similar to the proof of Theorem 14 in \cite{DFGenuines} using Algebraic Lemma for $\beta$ instead of Theorem 3 in \cite{DFCompositions} and using Theorem \ref{GenuineKahler} in place of Theorem 11 in \cite{DFGenuines}.
\end{proof}

We can now prove Theorem B.

\begin{proof}[Proof of Theorem B]
It follows from Theorem \ref{GenuineKahler} that along each connected component of an open dense subset of $M^{2n}$, either $f$ extend isometrically to a real K\"ahler submanifold, or $f$ is $D^d$-ruled with $d \geq 2n - 2p + 3l$. If the first case occurs the conclusion of the theorem is immediate. So we assume that $f$ is $D^{d}$-ruled.

Since $D^d$ is an asymptotic space we have that $A_\xi|_D (D) \subset  D^{\perp}$ for any $\xi \in T_f^{\perp}M$ and therefore \begin{equation}\label{eq:assymptoticShapeB}\dim (\ker A_\xi \cap D)\geq 2n+6l-4p,\end{equation} for all $\xi \in T_f^{\perp}M$. Also, from condition $(C_1)$ we have that $\Ker A_{\xi} = \Ker A_{\tau\xi}$ for all $\xi \in L^{2k}$, where $2k=l$. By definition of $D^d$, we have that $\Delta \cap J\Delta$ is the intersection of $D^d$ and $\bigcap_{i=1}^{k}\Ker A_{\xi_i}$ for some basis $\{\xi_1,\cdots, \xi_k, \tau\xi_1, \cdots \tau\xi_k\}$ of $L^{2k}$. Combining the inequalities \eqref{eq:assymptoticShapeB} for $\xi_1,\cdots,\xi_k$, we have that  $$\dim \Delta \cap J\Delta \geq d - k(2p-3l) \geq (2n-2p)+3l(k+1)-2pk.$$ Finally, observe that $3l(k+1)-2pk = 6k(k+1)-2pk \geq 0$ for $p \leq 6$ and the theorem follows.
\end{proof}

%{\bf Acknowledgements.} Felippe Guimar\~aes thanks IMPA for financial support and
%excellent research environment during the preparation of this paper.

\bibliographystyle{plain}
\bibliography{FirstLatexDoc_lib}

\vspace{1eM}
{\footnotesize
\begin{tabular}{lll}
	Felippe Guimar\~aes && Alcides de Carvalho J\'unior \\
	IME-USP && IMPA \\
	R. do Mat\~ao, 1010 && Est. Dona Castorina 110\\
	05508-090, S\~ao Paulo -- SP, Brazil && 22460-320, Rio de Janeiro -- RJ, Brazil \\
	\textit{e-mail:} \texttt{felippe@impa.br} && \textit{e-mail:} \texttt{alcidesj@impa.br}
\end{tabular}}
\end{document}